\documentclass{amsart}
\usepackage{amssymb}
\usepackage{amsmath,amsfonts,amssymb}
\usepackage{amsfonts, amsmath, amssymb, amscd, amsthm}

\def\bfabs{\par\vskip 4pt \noindent}
\newtheorem{theorem}{Theorem}[section]
\newtheorem{lemma}{Lemma}[section]
\newtheorem{proposition}{Proposition}[section]

\theoremstyle{remark}
\newtheorem{remark}{Remark}[section]

\begin{document}

\title[An optimal inequality on centroaffine
hypersurfaces] {An optimal inequality on locally strongly convex
centroaffine hypersurfaces}

\author{Xiuxiu Cheng and Zejun Hu}
\address{%
School of Mathematics and Statistics, Zhengzhou University,
Zhengzhou 450001, People's Republic of China.
\newline \indent{\it E-mail
addresses}: {\rm chengxiuxiu1988@163.com; huzj@zzu.edu.cn}}


\thanks{This project was supported by grants of NSFC-11371330.}

\date{}


\begin{abstract}
In this paper, we establish a general inequality for locally
strongly convex centroaffine hypersurfaces in $\mathbb{R}^{n+1}$
involving the norm of the covariant derivatives of both the
difference tensor $K$ and the Tchebychev vector field $T$. Our
result is optimal in that, applying our recent classification for
locally strongly convex centroaffine hypersurfaces with parallel
cubic form in \cite{CHM}, we can completely classify the
hypersurfaces which realize the equality case of the inequality.
\end{abstract}

\maketitle

\numberwithin{equation}{section}
\section{Introduction}\label{sect:1}

Let $\mathbb{R}^{n+1}$ be the $(n+1)$-dimensional affine space
equipped with its canonical flat connection $D$ and the parallel
volume form det. In this paper, we show that for locally strongly
convex centroaffine hypersurfaces in $\mathbb{R}^{n+1}$ there is an
optimal inequality involving centroaffine invariants.

Recall that in centroaffine differential geometry, we study
properties of hypersurfaces in $\mathbb{R}^{n+1}$ that are invariant
under the centroaffine transformation group $G$ in
$\mathbb{R}^{n+1}$. Here, by definition, $G$ is the subgroup of
affine transformation group in $\mathbb{R}^{n+1}$ which keeps the
origin $O\in \mathbb{R}^{n+1}$ invariant. Let $M^n$ be an
$n$-dimensional smooth manifold. An immersion $x: M^n\rightarrow
\mathbb{R}^{n+1}$ is said to be a centroaffine hypersurface if, for
each point $x\in M^n$, the position vector $x$ (from $O$) is
transversal to the tangent space $T_xM$ of $M$ at $x$. In that
situation, the position vector $x$ defines the {\it centroaffine
normalization} modulo orientation. For any vector fields $X$ and $Y$
tangent to $M^n$, we have the centroaffine formula of Gauss:
\begin{equation}\label{eqn:1.1}
D_{X}x_{*}(Y)=x_{*}(\nabla_{X}Y) + h(X,Y)(-\varepsilon x) ,
\end{equation}
where $\varepsilon=1$ or $-1$. Moreover, associated with
\eqref{eqn:1.1} we will call $-\varepsilon x$, $\nabla$ and $h$ the
centroaffine normal, the induced (centroaffine) connection and the
centroaffine metric, respectively. In this paper, we will consider
only locally strongly convex centroaffine hypersurfaces such that
the bilinear $2$-form $h$ defined by \eqref{eqn:1.1} remains
definite; then we will choose $\varepsilon$ such that the
centroaffine metric $h$ is positive definite.

Let $x: M^n\rightarrow \mathbb{R}^{n+1}$ be a locally strongly
convex centroaffine hypersurface and $\hat\nabla$ be the Levi-Civita
connection of its centroaffine metrc $h$. Then its difference tensor
$K$ is defined by $K(X,Y):=K_XY:=\nabla_XY-\hat{\nabla}_XY$; it is
symmetric as both connections are torsion free. Define the cubic
form $C$ by $C:=\nabla h$; it is related to the difference tensor by
the equation
\begin{equation}\label{eqn:1.2}
C(X,Y,Z):=(\nabla_Xh)(Y,Z)=-2h(K_XY,Z).
\end{equation}
It follows that $C$ is a totally symmetric tensor of type $(0,3)$,
and that $\hat{\nabla}K=0$ is equivalent to $\hat{\nabla}C=0$. Now,
we define the Tchebychev form $T^\sharp$ and its associated
Tchebychev vector field $T$ such that:
\begin{equation}\label{eqn:1.3}
nT^\sharp(X)={\rm trace}\,(K_{X}),\ \ h(T, X)=T^\sharp(X).
\end{equation}
If $T=0$, or equivalently, ${\rm trace}\,K_X=0$ for any tangent
vector $X$, then $M^n$ reduced to be the so-called {\it proper
(equi-)affine hypersphere} centered at the origin $O$ (cf. p.279 of
\cite{LSZH}, or see Section 1.15.2-3 therein for more details).
Using the difference tensor $K$ and the Tchebychev vector field $T$
one can define a traceless difference tensor $\tilde{K}$ by
\begin{equation}\label{eqn:1.4}
\tilde{K}(X,Y):=K(X,Y)-\tfrac{n}{n+2}\big[h(X,Y)T+h(X,T)Y
+h(Y,T)X\big].
\end{equation}

It is well-known that $\tilde{K}$ vanishes if and only if $x(M^n)$
lies in a hyperquadric (cf. Section 7.1 in \cite{SSV}; Lemma 2.1 and
Remark 2.2 in \cite{LLSSW}; refer also to \cite{C} and its
reviewer's comments in MR2155181).

Now, we can state the main result of this paper as follows:

\begin{theorem}\label{thm:1.1}
Let $x: M^n\rightarrow \mathbb{R}^{n+1}$ be a locally strongly
convex centroaffine hypersurface. Then the difference tensor $K$ and
the Tchebychev vector field $T$ of $M^n$ satisfy the following
inequality
\begin{equation}\label{eqn:1.5}
\|\hat{\nabla}K\|^2\geq \tfrac{3n^2}{n+2}\|\hat{\nabla}T\|^2,
\end{equation}
where $\|\cdot\|$ denotes the tensorial norm with respect to the
centroaffine metric $h$. Moreover, the equality holds at every point
of $M^n$ if and only if $\hat{\nabla}\tilde{K}=0$, and one of the
following cases occurs:
\begin{enumerate}
\item[(i)] $x(M^n)$ is an open part of a locally strongly
convex hyperquadric; or
\vskip 1mm
\item[(ii)] $x(M^n)$ is obtained as the (generalized) Calabi product of a
lower dimensional locally strongly convex centroaffine hypersurface
with parallel cubic form and a point; or
\vskip 1mm
\item[(iii)] $x(M^n)$ is obtained as the (generalized) Calabi product of two
lower dimensional locally strongly convex centroaffine hypersurfaces
with parallel cubic form; or
\vskip 1mm
\item[(iv)] $n=\tfrac{1}{2}m(m+1)-1,\ m\ge3$, $x(M^n)$ is centroaffinely equivalent to the standard
embedding of
$\mathrm{SL}(m,\mathbb{R})/\mathrm{SO}(m)\hookrightarrow
\mathbb{R}^{n+1}$;
or
\vskip 1mm
\item[(v)] $n=\tfrac{1}{4}(m+1)^2-1,\ m\ge5$,
$x(M^n)$ is centroaffinely equivalent to the standard embedding
$\mathrm{SL}(\tfrac{m+1}{2},\mathbb{C})/\mathrm{SU}(\tfrac{m+1}{2})
\hookrightarrow\mathbb{R}^{n+1}$;
or
\vskip 1mm
\item[(vi)] $n=\tfrac{1}{8}(m+1)(m+3)-1,\ m\ge9$,
$x(M^n)$ is centroaffinely equivalent to the standard embedding
$\mathrm{SU}^*(\tfrac{m+3}{2})/\mathrm{Sp}(\tfrac{m+3}{4})\hookrightarrow
\mathbb{R}^{n+1}$;
or
\vskip 1mm
\item[(vii)] $n=26$,
$x(M^n)$ is centroaffinely equivalent to the standard embedding\\
$\mathrm{E}_{6(-26)}/\mathrm{F}_4\hookrightarrow\mathbb{R}^{27}$; or
\vskip 1mm
\item[(viii)] $x(M^n)$ is locally centroaffinely equivalent to the
canonical centroaffine hypersurface
$x_{n+1}=\tfrac{1}{2x_{1}}\sum_{k=2}^{n}x_{k}^{2} + x_{1}\ln x_{1}$.
\end{enumerate}

\end{theorem}
\begin{remark}\label{rm:1.1}
For detailed discussions about all the above examples, namely the
notion of {\it (generalized) Calabi product} and the standard
embedding, the readers are referred to \cite{CHM} (cf. also
\cite{HLV}). We should point it out that the ellipsoids and the
hyperboloids which are centered at the origin $O$, and also the
hypersurfaces in (ii)-(viii), have parallel cubic form, i.e.,
$\hat{\nabla}C=0$ or equivalently $\hat{\nabla}K=0$; while a
hyperquadric with no center or not being centered at the origin $O$
has the properties that $K\not=0$ and $\hat{\nabla}K\not=0$ (cf.
\cite{CHM}). We also remark that a centroaffine hypersurface is
called {\it canonical} meaning that its centroaffine metric $h$ is
flat and its cubic form $C$ satisfies $\hat{\nabla}C=0$ (cf.
\cite{LW}).
\end{remark}

%
\begin{remark}\label{rm:1.2}
The lists of centroaffine hypersurfaces as shown in Theorem
\ref{thm:1.1} give the classification of centroaffine hypersurfaces
in $\mathbb{R}^{n+1}$ with parallel traceless cubic form (which is
equivalent to $\hat{\nabla}\tilde{K}=0$) for every $n\ge2$. This is
a complete extension of \cite{LW2} where the classification was
achieved only for $n=2$. On the other hand, locally strongly convex
centroaffine hypersurfaces with $\hat{\nabla}K=0$ are classified in
\cite{CHM} for every dimensions.
\end{remark}

%
\begin{remark}\label{rm:1.3}
Besides that as stated in \cite{CHM}, different characterizations on
the typical examples of centroaffine hypersurfaces appearing in
Theorem \ref{thm:1.1} were established in our recent articles,
\cite{CH} and \cite{CHLL}, from other aspects of differential
geometric invariants.
\end{remark}

%
\begin{remark}\label{rm:1.4}
Related with the study of centroaffine hypersurfaces, with pleasure
we would like to introduce the interesting results of Li, Simon and
Zhao \cite{LSZ} and also the very recent development due to
Cort\'es, Nardmann and Suhr \cite{CNS}, where among other important
results the authors investigated the problem under what conditions a
locally strongly convex centroaffine hypersurface is complete with
respect to the centroaffine metric.
\end{remark}


\bfabs{\bf Acknowledgements}. The authors would like express their
thanks to Professors H. Li, U. Simon and L. Vrancken for many
aspects of their help with this paper. As a matter of fact, our
result Theorem \ref{thm:1.1} could be regarded as an affine
differential geometric counterpart of the main result in \cite{LV},
where Li and Vrancken proved a basic inequality for Lagrangian
submanifolds in complex space forms and as its direct consequence
they obtained a new characterization of the Whitney spheres.


\numberwithin{equation}{section}
\section{Preliminaries}\label{sect:2}

In this section, we briefly recall some basic facts about
centroaffine hypersurfaces. We refer to \cite{LLS}, \cite{LSZH,NS},
\cite{SSV} and \cite{LW,W} for more detailed discussions.

Given a centroaffine hypersurface $M^n$, we choose an
$h$-orthonormal tangential frame field $\{e_1, \ldots, e_n\}$. Let
$\{\theta_1,\ldots,\theta_n\}$ be its dual frame field and
$\{\theta_{ij}\}$ its Levi-Civita connection forms. Let $K_{ij}^k$
and $T^i$ denote the components of $K$ and $T$ with respect to
$\{e_i\}$. Then \eqref{eqn:1.4} can be written as
\begin{equation}\label{eqn:2.1}
\tilde{K}_{ij}^k:=K_{ij}^k-\tfrac{n}{n+2}(T^k \delta_{ij} +T^i
\delta_{jk}+T^j \delta_{ik}),
\end{equation}
where $K_{ij}^k=h(K_{e_i}e_j,e_k)$, $T^i=\tfrac1n\sum_jK^i_{jj}$.

Let $K_{ij,l}^k$ and $T^j_{,i}$ be the components of the covariant
differentiation $\hat\nabla K$ and $\hat\nabla T$, respectively,
which by definition can be expressed by
$$
\sum_lK_{ij,l}^k\theta_l=dK_{ij}^k+\sum_lK_{lj}^k\theta_{li}+\sum_lK_{il}^k\theta_{lj}+\sum_lK_{ij}^l\theta_{lk},
$$
$$
\sum_iT_{,i}^j\theta_i=dT^j+\sum_iT^i\theta_{ij}.
$$

Denote by $\hat{R}_{ijkl}$ the components of the Riemannian
curvature tensor of the centroaffine metric $h$. Then, we have the
equations of Gauss and Codazzi as follows:
\begin{equation}\label{eqn:2.3}
\hat{R}_{ijkl}=\varepsilon(\delta_{ik}\delta_{jl}
-\delta_{il}\delta_{jk}) +\sum_m(K_{il}^mK_{jk}^m-K_{ik}^mK_{jl}^m),
\end{equation}
\begin{equation}\label{eqn:2.2}
K_{ij,l}^k=K_{ij,k}^l,\ \ 1\le i,j,k,l\le n.
\end{equation}

\vskip 3mm



\numberwithin{equation}{section}
\section{The inequality and some related lemmas}\label{sect:3}

We start with the following result.

\begin{proposition}\label{prop:3.1}
Let $x: M^n\rightarrow \mathbb{R}^{n+1}$ be a locally
strongly convex centroaffine hypersurface. Then
\begin{equation}\label{eqn:3.1}
\|\hat{\nabla}K\|^2\geq \tfrac{3n^2}{n+2}\|\hat{\nabla}T\|^2,
\end{equation}
where $\|\hat{\nabla}K\|^2=\sum (K_{ij,l}^k)^2$,
$\|\hat{\nabla}T\|^2=\sum (T_{,i}^j)^2$. Moreover, the equality
holds in \eqref{eqn:3.1} if and only if the traceless difference
tensor $\tilde{K}$ is parallel, i.e., $\hat\nabla\tilde{K}=0$, or
equivalently:
\begin{equation}\label{eqn:3.2}
K_{ij,l}^k=\tfrac{n}{n+2}(T_{,l}^k \delta_{ij} +T_{,l}^i
\delta_{jk}+T_{,l}^j \delta_{ik}),\ \ 1\le i,j,k,l\le n.
\end{equation}
\end{proposition}

\begin{proof}
From the definition \eqref{eqn:1.4} or \eqref{eqn:2.1}, we have
\begin{equation}\label{eqn:3.3}
\tilde{K}_{ij,l}^k=K_{ij,l}^k-\tfrac{n}{n+2}(T_{,l}^k \delta_{ij}
+T_{,l}^i \delta_{jk}+T_{,l}^j \delta_{ik}).
\end{equation}

It is easy to check that
\begin{equation}\label{eqn:3.4}
\begin{aligned}
0\le\|\hat{\nabla}\tilde{K}\|^2:&=\sum(\tilde{K}_{ij,l}^k)^2
=\sum (K_{ij,l}^k)^2-\tfrac{3n^2}{n+2}\sum (T_{,i}^j)^2\\
&=\|\hat{\nabla}K\|^2-\tfrac{3n^2}{n+2}\|\hat{\nabla}T\|^2.
\end{aligned}
\end{equation}

Obviously, equality in \eqref{eqn:3.1} holds if and only if
$\|\hat{\nabla}\tilde{K}\|=0$, i.e., it holds
$\tilde{K}_{ij,l}^k=0$, $1\le i,j,k,l\le n$, which is equivalent to
\eqref{eqn:3.2}.
\end{proof}

Next, we investigate the implications if \eqref{eqn:3.2} holds.

\begin{lemma}\label{lm:3.2}
Let $x: M^n\rightarrow \mathbb{R}^{n+1}$ be a locally strongly
convex centroaffine hypersurface. If \eqref{eqn:3.2} holds, then we
have $T_{,k}^j=\tfrac{1}{n}\sum T_{,i}^i \delta_{jk},\ 1\le j,k\le
n$, namely,
\begin{equation}\label{eqn:3.5}
\hat{\nabla}T=\lambda \cdot {\rm id},\ \ \lambda=\tfrac{1}{n}\sum
T_{,i}^i.
\end{equation}
\end{lemma}
\begin{proof}
Exchanging $k$ with $l$ in \eqref{eqn:3.2}, we have
\begin{equation}\label{eqn:3.6}
K_{ij,k}^l=\tfrac{n}{n+2}(T_{,k}^l \delta_{ij}
+T_{,k}^i \delta_{jl}+T_{,k}^j \delta_{il}).
\end{equation}

Combining \eqref{eqn:2.2}, \eqref{eqn:3.2} and \eqref{eqn:3.6}, we
obtain
\begin{equation}\label{eqn:3.7}
T_{,k}^l \delta_{ij}
+T_{,k}^i \delta_{jl}+T_{,k}^j \delta_{il}=T_{,l}^k \delta_{ij}
+T_{,l}^i \delta_{jk}+T_{,l}^j \delta_{ik},\ 1\le i,j,k,l\le n.
\end{equation}

Taking the summation for $i=l$ in \eqref{eqn:3.7} and noting that
$T_{,k}^j=T^k_{,j}$, we get
\begin{equation}\label{eqn:3.8}
T_{,k}^j=\tfrac{1}{n}\sum T_{,i}^i \delta_{jk},\ 1\le j,k\le n.
\end{equation}
This verifies the assertion.
\end{proof}

\begin{remark}\label{rm:3.1}
If \eqref{eqn:3.5} holds, then $T$ is a conformal vector field and
$M^n$ by definition is called a Tchebychev hypersurface (cf.
\cite{LW1}). Therefore, if the equality holds in \eqref{eqn:3.1}
then $M^n$ is a Tchebychev hypersurface.
\end{remark}

\begin{lemma}\label{lm:3.3}
Let $x: M^n\rightarrow \mathbb{R}^{n+1}$ be a locally strongly
convex centroaffine hypersurface. Then, \eqref{eqn:3.2} holds if and
only if it holds that
\begin{equation}\label{eqn:3.9}
K_{ij,l}^k=\mu(\delta_{kl} \delta_{ij} +\delta_{il}
\delta_{jk}+\delta_{jl} \delta_{ik}),\ \ 1\le i,j,k,l\le n,
\end{equation}
where $\mu=\tfrac{1}{n+2}\sum T_{,l}^l$.
%
%
\end{lemma}
\begin{proof}
For the ``if" part, we assume that \eqref{eqn:3.9} holds. By summing
over $i=j$ in \eqref{eqn:3.9}, we get
\begin{equation}\label{eqn:3.11}
T_{,l}^k=\tfrac{n+2}{n}\mu \delta_{kl},\ 1\le k,l\le n.
\end{equation}
Then \eqref{eqn:3.2} immediately follows.

Conversely, for the ``only if" part, we assume that \eqref{eqn:3.2}
holds. Then we have \eqref{eqn:3.8}, and therefore \eqref{eqn:3.9}
holds with $\mu=\tfrac1{n+2}\sum_lT_{,l}^l$.
\end{proof}

Now, we fix a point $p\in M^n$. For subsequent purpose, we will
review the well-known construction of a typical orthonormal basis
with respect to the centroaffine metric $h$ for $T_pM^n$, which was
introduced by Ejiri and has been widely applied, and proved to be
very useful for various situations, see e.g. \cite{HLSV} and
\cite{LV,MU}. The idea is to construct from the $(1,2)$ tensor $K$ a
self adjoint operator at a point; then one extends the eigenbasis to
a local field.

Let $p\in M^n$ and $U_pM^n=\{u\in T_pM^n \mid h(u,u)=1\}$. Since
$M^n$ is locally strongly convex, $U_pM^n$ is compact. We define a
function $f$ on $U_pM^n$ by $f(u)=h(K_uu,u)$. Then there is an
element $e_{1}\in U_pM^n$ at which the function $f(u)$ attains an
absolute maximum, denoted by $\lambda_{1}$. Then we have the
following lemma. For its proof, we refer the reader to \cite{HLSV}.
\begin{lemma}[\cite{HLSV}]\label{lm:3.4}
There exists an orthonormal basis $\{e_1,\ldots,e_n\}$ of $T_pM^n$
such that the following hold:
\begin{enumerate}
\item[(i)] $K_{e_{1}}e_{i}=\lambda_ie_i,\ for \ i=1,\ldots,n$.
\item[(ii)] $\lambda_1\geq 2\lambda_i$, for
$i\geq 2$. If $\lambda_1= 2\lambda_i$,
then $f(e_i)=0$.
\end{enumerate}
\end{lemma}

When working at the point $p\in
M^n$, we will always assume that an orthonormal basis is chosen such
that Lemma \ref{lm:3.4} is satisfied. While if we work at a
neighborhood of $p\in M^n$ and if not stated otherwise, we will
choose an $h$-orthonormal frame field $\{E_1,\ldots, E_n\}$ such
that $E_1(p)=e_1, \ldots, E_n(p)=e_n$, and $\{e_1,\ldots,e_n\}$ is
chosen as in Lemma \ref{lm:3.4}.

\vskip 2mm

The following lemma is crucial for our proof of Theorem
\ref{thm:1.1}.

\begin{lemma}\label{lm:3.5}
Let $x: M^n\rightarrow \mathbb{R}^{n+1}$ be a locally strongly
convex centroaffine hypersurface. If \eqref{eqn:3.9} holds, then we
have
\begin{equation}\label{eqn:3.12}
e_l(\mu)=0,\ \ 2\le l\le n,
\end{equation}
\begin{equation}\label{eqn:3.13}
e_1(\mu)=(2\lambda_l-\lambda_1)(\lambda_l^2-\lambda_1\lambda_l
+\varepsilon),\ \ 2\le l\le n,
\end{equation}
\begin{equation}\label{eqn:3.14}
(2\lambda_k-\lambda_1)(\lambda_l-\lambda_j)K_{jk}^l=0,\ \ 2\le j\neq
l\le n,\ 1\le k\le n,
\end{equation}
\begin{equation}\label{eqn:3.15}
(\lambda_l^2-\lambda_1\lambda_l +\varepsilon)K_{ll}^l=0,\ \ 2\le
l\le n.
\end{equation}
\end{lemma}
\begin{proof}

Taking the covariant derivative of \eqref{eqn:3.9} implies that
\begin{equation}\label{eqn:3.16}
K_{ij,lq}^k=e_q(\mu)(\delta_{kl} \delta_{ij} +\delta_{il}
\delta_{jk}+\delta_{jl} \delta_{ik}).
\end{equation}
Exchanging $l$ with $q$ in \eqref{eqn:3.16}, we have
\begin{equation}\label{eqn:3.17}
K_{ij,ql}^k=e_l(\mu)(\delta_{kq}\delta_{ij}+\delta_{iq}
\delta_{jk}+\delta_{jq}\delta_{ik}).
\end{equation}

From \eqref{eqn:3.16}, \eqref{eqn:3.17} and the Ricci identity, we
use \eqref{eqn:2.3} to obtain
\begin{equation}\label{eqn:3.18}
\begin{aligned}
&e_q(\mu)(\delta_{kl} \delta_{ij} +\delta_{il}
\delta_{jk}+\delta_{jl} \delta_{ik})-e_l(\mu)(\delta_{kq}
\delta_{ij}
+\delta_{iq} \delta_{jk}+\delta_{jq} \delta_{ik})\\
&=\sum_m
K_{mj}^k\big[\varepsilon(\delta_{iq}\delta_{lm}-\delta_{il}\delta_{qm})
+\sum_r(K_{il}^rK_{qm}^r-K_{qi}^rK_{lm}^r)\big]\\
&\ \ +\sum_m
K_{mi}^k\big[\varepsilon(\delta_{jq}\delta_{lm}-\delta_{jl}\delta_{qm})
+\sum_r(K_{jl}^rK_{qm}^r-K_{qj}^rK_{lm}^r)\big]\\
&\ \ -\sum_m
K_{ij}^m\big[\varepsilon(\delta_{mq}\delta_{kl}-\delta_{ml}\delta_{kq})
+\sum_r(K_{ml}^rK_{kq}^r-K_{mq}^rK_{lk}^r)\big].
\end{aligned}
\end{equation}

Taking $i=j=q=1$ and $l\geq 2$ in \eqref{eqn:3.18}, we obtain that
\begin{equation}\label{eqn:3.19}
e_1(\mu)\delta_{kl}-3e_l(\mu)\delta_{k1}=(2\lambda_k
-\lambda_1)(\lambda_k^2-\lambda_1\lambda_k
+\varepsilon)\delta_{kl}.
\end{equation}

First, letting $k=1$ in \eqref{eqn:3.19}, we get
\begin{equation*}
e_l(\mu)=0,\ \ 2\le l\le n.
\end{equation*}

Next, letting $k=l\geq 2$ in \eqref{eqn:3.19}, we have
\begin{equation*}
e_1(\mu)=(2\lambda_l-\lambda_1)(\lambda_l^2-\lambda_1\lambda_l
+\varepsilon),\ \ 2\le l\le n.
\end{equation*}

Then, letting $i=j=1$ and $2\le q\neq l \le n$ in \eqref{eqn:3.18},
combining with \eqref{eqn:3.12}, we obtain
\begin{equation*}
(2\lambda_k-\lambda_1)(\lambda_l-\lambda_q)K_{qk}^l=0,\ \ 2\le q\neq
l\le n,\ 1\le k\le n.
\end{equation*}

Finally, letting $i=q=1$ and $j=k=l\geq 2$ in \eqref{eqn:3.18}, a
direct calculation gives
\begin{equation*}
(\lambda_l^2-\lambda_1\lambda_l +\varepsilon)K_{ll}^l=0,\ \ 2\le
l\le n.
\end{equation*}

We have completed the proof of Lemma \ref{lm:3.5}.
\end{proof}

\vskip 2mm


\numberwithin{equation}{section}
\section{Proof of the main theorem}\label{sect:4}

In this section, we will complete the proof of Theorem
\ref{thm:1.1}. Let $x: M^n\rightarrow \mathbb{R}^{n+1}$ be a locally
strongly convex centroaffine hypersurface. Then, according to
Proposition \ref{prop:3.1} and Lemma \ref{lm:3.3}, to prove Theorem
\ref{thm:1.1} we are left to consider the case that \eqref{eqn:3.9}
holds identically for some function $\mu$ on $M^n$.

\subsection{\eqref{eqn:3.9} holds with $\mu\neq {\rm
constant}$}\label{sect:4.1}

In this subsection, we consider $n$-dimensional locally strongly
convex centroaffine hypersurfaces such that \eqref{eqn:3.9} holds
identically with $\mu\not={\rm constant}$. Since our result is local
in nature, the non-constancy of $\mu$ allows us to assume that
$U_1:=\{ q\in M^n\mid X(\mu)=0, \forall X\in T_qM^n\}$ is not an
open subset. Therefore, from now on we will carry our discussion in
the following open dense subset of $M^n$:
$$
M'=\{q\in M^n\mid\ {\rm there\ exists}\ X\in T_qM^n\ {\rm such\
that\ } X(\mu)\neq0\}.
$$

First of all, we have the following lemma.
\begin{lemma}\label{lm:4.1}
If \eqref{eqn:3.9} holds at every point of $M^n$ with $\mu\neq {\rm
constant}$, then with respect to the orthonormal basis as stated in
Lemma \ref{lm:3.4}, the number of the distinct eigenvalues of
$K_{e_1}$ can be at most $3$, so that it equals $2$ or $3$.
\end{lemma}
\begin{proof}
Let $p\in M'$. From Lemma \ref{lm:3.5}, we have
\begin{equation}\label{eqn:4.1}
e_1(\mu)=(2\lambda_l-\lambda_1)(\lambda_l^2-\lambda_1\lambda_l
+\varepsilon)\neq 0,\ \ 2\le l\le n.
\end{equation}
It follows that
\begin{equation}\label{eqn:4.2}
\lambda_1>0,\ \lambda_1-2\lambda_l>0,\
\lambda_l^2-\lambda_1\lambda_l +\varepsilon\neq0,\ \ 2\le l\le n,
\end{equation}
and, for each $2\le l\le n$, $y_l=\lambda_1-2\lambda_l$ satisfies
the following equation in $y$:
\begin{equation}\label{eqn:4.3}
e_1(\mu)+\tfrac{1}{4}y(y^2+4\varepsilon-\lambda_1^2)=0,\ \ y>0.
\end{equation}

Now, about the solution $y$ of \eqref{eqn:4.3}, we consider the
following three cases:

\begin{enumerate}
\item[(1)] If $4\varepsilon-\lambda_1^2\geq 0$, then \eqref{eqn:4.3} shows that
$e_1(\mu)<0$. In this case, as an equation of $y$, \eqref{eqn:4.3}
has only one positive solution. This implies that we have
$\lambda_2=\cdots=\lambda_n$.

\vskip 1mm

\item[(2)] If $4\varepsilon-\lambda_1^2<0$ and $e_1(\mu)<0$,
then again \eqref{eqn:4.3} has only one positive solution $y$ and
that $\lambda_2=\cdots=\lambda_n$.

\vskip 1mm

\item[(3)] If $4\varepsilon-\lambda_1^2<0$ and $e_1(\mu)>0$, then
\eqref{eqn:4.3} has at most two positive solutions. This implies
that at most two of $\{\lambda_2,\ldots,\lambda_n\}$ are distinct.
\end{enumerate}

On the other hand, from \eqref{eqn:4.2} we easily see that
$\lambda_1>\lambda_l$ for all $l\ge2$.

This clearly completes the proof of Lemma \ref{lm:4.1}.
\end{proof}

As a direct consequence of Lemma \ref{lm:4.1}, the study of
centroaffine hypersurfaces such that \eqref{eqn:3.9} holds
identically with $\mu\neq {\rm constant}$ can be divided into two
cases:

\vskip 2mm

\textbf{Case (i)}.\ \ $\lambda_{2}=\cdots=\lambda_{m}<
\lambda_{m+1}=\cdots=\lambda_{n},\ \ 2\leq m\leq n-1$.

\vskip 1mm

\textbf{Case (ii)}.\ \ $\lambda_{2}=\cdots=\lambda_{n}$.

\vskip 2mm

The following lemma is important in sequel of this subsection.

\begin{lemma}\label{lm:4.2}
If \eqref{eqn:3.9} holds at every point of $M^n$ with $\mu\neq {\rm
constant}$, then, for $\{e_i\}$ as described in Lemma \ref{lm:3.4},
the difference tensor $K$ takes the following form:
\begin{equation}\label{eqn:4.4}
K_{e_1}e_1=\lambda_1 e_1,\ K_{e_1}e_i=\lambda_i e_i,\
K_{e_i}e_j=\lambda_i\delta_{ij} e_1,\ \ i, j=2,\ldots,n,
\end{equation}
\end{lemma}
\begin{proof} We separate the proof into two cases as above.

If {\bf Case (i)} occurs, then from \eqref{eqn:3.15} and
\eqref{eqn:4.2}, we get
$$
h(K_{e_l}e_l,e_l)=0,\ \ 2\le l\le m.
$$
It follows that
\begin{equation}\label{eqn:4.5}
h(K_{e_i}e_j,e_k)=0,\ \ 2\leq i, j, k\leq m.
\end{equation}

On the other hand, from \eqref{eqn:3.14} and \eqref{eqn:4.2}, we
obtain
\begin{equation}\label{eqn:4.6}
h(K_{e_i}e_j,e_k)=0,\ \ 2\leq i, j\leq m,\ m+1\leq k\leq n.
\end{equation}

Combining \eqref{eqn:4.5}, \eqref{eqn:4.6} and the fact
$h(K_{e_i}e_j,e_1)=\lambda_i\delta_{ij}$, we get the assertion
\begin{equation}\label{eqn:4.7}
K_{e_i}e_j=\lambda_i\delta_{ij}e_1,\ \ 2\leq i, j\leq m.
\end{equation}

Similarly, we can prove that
\begin{equation}\label{eqn:4.8}
K_{e_i}e_j=\lambda_i\delta_{ij}e_1,\ \ m+1\leq i, j\leq n.
\end{equation}

From \eqref{eqn:3.14} and \eqref{eqn:4.2} again, we have
\begin{equation}\label{eqn:4.9}
h(K_{e_i}e_j,e_k)=0,\ \ 2\leq i\leq m,\ m+1\leq j\leq n,\ 1\leq
k\leq n.
\end{equation}
This shows that
$$
K_{e_i}e_j=0,\ \ 2\leq i\leq m,\ m+1\leq j\leq n.
$$

In summary, we have completed the proof of Lemma \ref{lm:4.2} for
{\bf Case (i)}.

Next, similar to the proof of \eqref{eqn:4.5}, we can verify the
assertion for {\bf Case (ii)}.
\end{proof}

To treat the above two cases separately, we first state the
following result.
\begin{lemma}\label{lm:4.3}
Case (i) does not occur.
\end{lemma}
\begin{proof}
Suppose on the contrary that {\bf Case (i)} does occur. Then, from
\eqref{eqn:3.12} and \eqref{eqn:4.1}, we get
\begin{equation}\label{eqn:4.10}
h({\rm grad}\,\mu,e_1)=e_1(\mu)\neq0,\ h({\rm
grad}\,\mu,e_l)=e_l(\mu)=0,\ \ 2\leq l\leq n.
\end{equation}
It follows that $e_1=\pm\tfrac{{\rm grad}\,\mu}{\|{\rm
grad}\,\mu\|}(p)$. Without loss of generality, we assume that
$e_1=\tfrac{{\rm grad}\,\mu}{\|{\rm grad}\,\mu\|}(p)$.

Now, in a neighborhood $U$ around $p$, we define a unit vector field
$E_1=\tfrac{{\rm grad}\,\mu}{\|{\rm grad}\,\mu\|}$. It is easily
seen from the proof of \eqref{eqn:4.10} that, for each $q\in U$, the
function $f$ should achieve its absolute maximum over $U_qM^n$
exactly at $E_1(q)$. Furthermore, the continuity of eigenvalue
functions of $K_{E_1}$ (cf. \cite{S}) and Lemma \ref{lm:4.1} imply
that the multiplicity of each of its eigenvalue functions is
constant. Then applying Lemma 1.2 of \cite{S} we have a smooth
eigenvector extension of $K_{E_1}$, from $\{e_1,e_2,\ldots,e_n\}$ at
$p$ to $\{E_1(q),E_2(q),\ldots,E_n(q)\}$ at any point $q$ in a
neighborhood of $p$, such that $K_{E_1}E_i=\tilde{\lambda}_iE_i$,
with the functions $\{\tilde{\lambda}_i\}_{i=1}^n$ satisfying
$\tilde{\lambda}_1\ge2\tilde{\lambda}_i$ for $i\ge2$ and
$$
\tilde{\lambda}_{2}=\cdots=\tilde{\lambda}_{m}<
\tilde{\lambda}_{m+1}=\cdots=\tilde{\lambda}_{n},\ \ 2\leq m\leq
n-1.
$$

It is easy to see that, with respect to the local $h$-orthonormal
frame field $\{E_i\}_{i=1}^n$ and the eigenvalue functions
$\{\tilde\lambda_i\}_{i=1}^n$, the foregoing lemmas that from Lemma
\ref{lm:3.4} up to Lemma \ref{lm:4.2} remain valid.

Now, applying Lemma \ref{lm:4.2}, we obtain
\begin{equation}\label{eqn:4.12}
\begin{aligned}
(\hat{\nabla}_{E_i}K)(E_1,E_1)&=\hat{\nabla}_{E_i}K(E_1,E_1)-2K(\hat{\nabla}_{E_i}E_1,E_1)\\
&=\hat{\nabla}_{E_i}\tilde{\lambda}_1E_1-2\sum_{k=2}^n\tilde{\lambda}_kh(\hat{\nabla}_{E_i}E_1,E_k)E_k\\
&=E_i(\tilde{\lambda}_1)E_1+\sum_{k=2}^n(\tilde{\lambda}_1-2\tilde{\lambda}_k)h(\hat{\nabla}_{E_i}E_1,E_k)E_k,\
\ i\geq 2,
\end{aligned}
\end{equation}
and
\begin{equation}\label{eqn:4.13}
\begin{aligned}
(\hat{\nabla}_{E_i}K)(E_i,E_i)&=\hat{\nabla}_{E_i}K(E_i,E_i)-2K(\hat{\nabla}_{E_i}E_i,E_i)\\
&=\hat{\nabla}_{E_i}\tilde{\lambda}_iE_1-2\tilde{\lambda}_ih(\hat{\nabla}_{E_i}E_i,E_1)E_i
-2\tilde{\lambda}_ih(\hat{\nabla}_{E_i}E_i,E_i)E_1\\
&=E_i(\tilde{\lambda}_i)E_1+3\tilde{\lambda}_ih(\hat{\nabla}_{E_i}E_1,E_i)E_i\\
&\hspace{19mm}+\sum_{k\neq
i}\tilde{\lambda}_ih(\hat{\nabla}_{E_i}E_1,E_k)E_k,\ \ i\geq 2.
\end{aligned}
\end{equation}

Then, from \eqref{eqn:3.9}, \eqref{eqn:4.12}, \eqref{eqn:4.13} and
the definition of $K_{ij,l}^k$, we obtain that
\begin{equation}\label{eqn:4.14}
\mu=K_{11,i}^i=(\tilde{\lambda}_1-2\tilde{\lambda}_i)h(\hat{\nabla}_{E_i}E_1,E_i),\
\ 2\le i\le n,
\end{equation}
\begin{equation}\label{eqn:4.15}
3\mu=K_{ii,i}^i=3\tilde{\lambda}_ih(\hat{\nabla}_{E_i}E_1,E_i),\ \
2\le i\le n.
\end{equation}

From \eqref{eqn:4.14}, \eqref{eqn:4.15}, and noting that
$\tilde{\lambda}_1-2\tilde{\lambda}_i\neq 0$ for $2\le i\le n$, we
finally get
\begin{equation}\label{eqn:4.16}
\tilde{\lambda}_1=3\tilde{\lambda}_i,\ \ 2\le i\le n.
\end{equation}

Hence, we have $\tilde{\lambda}_2=\cdots=\tilde{\lambda}_n$. This is
a contradiction to {\bf Case (i)}.
\end{proof}

According to Lemmas \ref{lm:4.1} and \ref{lm:4.3}, we see that if
\eqref{eqn:3.9} holds at every point of $M^n$ with $\mu\neq {\rm
constant}$, then {\bf Case (ii)} should occur at every point of
$M'$. Moreover, we can prove the following lemma.

\begin{lemma}\label{lm:4.4}
If \eqref{eqn:3.9} holds at every point of $M^n$ with $\mu\neq {\rm
constant}$, then there exists a local $h$-orthonormal frame field
$\{E_1,\ldots,E_{n}\}$ and a smooth non-vanishing function $\lambda$
such that the difference tensor $K$ takes the following form:
\begin{equation}\label{eqn:4.17}
K_{E_1}E_1=3\lambda E_1,\ K_{E_1}E_i=\lambda E_i,\
K_{E_i}E_j=\lambda\delta_{ij}E_1,\ \ i, j=2,\ldots,n.
\end{equation}
\end{lemma}
\begin{proof}
First of all, we see that \eqref{eqn:4.10} still holds, and without
loss of generality we may assume that $e_1=\tfrac{{\rm
grad}\,\mu}{\|{\rm grad}\,\mu\|}(p)$. Now we define $E_1=\tfrac{{\rm
grad}\,\mu}{\|{\rm grad}\,\mu\|}$. Similar as in the proof of Lemma
\ref{lm:4.3}, for each point $q$ in a neighborhood $U$ of $p$, the
function $f$ should achieve its absolute maximum over $U_qM^n$
exactly at $E_1(q)$. Moreover, due to that $K_{E_1}(q)$ has exactly
two distinct eigenvalues with multiplicities $1$ and $n-1$,
respectively, we can apply Lemma 1.2 of \cite{S} again to obtain
local orthonormal eigenvector fields of $K_{E_1}$, extending from
$\{e_1,e_2,\ldots,e_n\}$ at $p$ to $\{E_1,E_2,\ldots,E_n\}$ around
$p$, such that $K_{E_1}E_i=\tilde{\lambda}_iE_i$, with the
eigenvalue functions $\{\tilde{\lambda}_i\}_{i=1}^n$ satisfy
$\tilde{\lambda}_{2}=\cdots=\tilde{\lambda}_n$.

It is easily seen that, with respect to $\{E_i\}_{i=1}^n$ and
$\{\tilde\lambda_i\}_{i=1}^n$, the foregoing lemmas, from Lemma
\ref{lm:3.4} up to Lemma \ref{lm:4.2}, and that the equations from
\eqref{eqn:4.12} up to \eqref{eqn:4.16}, are still valid. Hence, we
have $\tilde\lambda_1=3\tilde\lambda_i$ for $i\ge2$.

This completes the proof of Lemma \ref{lm:4.4}.
\end{proof}

\vskip 2mm


\numberwithin{equation}{section}
\subsection{\eqref{eqn:3.9} holds with $\mu={\rm
constant}$}\label{sect:4.2}

In this subsection, we consider $n$-dimensional locally strongly
convex centroaffine hypersurfaces such that \eqref{eqn:3.9} holds
identically with $\mu={\rm constant}$. The following Proposition is
the main result of this subsection.

\begin{proposition}\label{pr:5.1}
Let $x: M^n\rightarrow \mathbb{R}^{n+1}$ be a locally strongly
convex centroaffine hypersurface. If \eqref{eqn:3.9} holds at every
point of $M^n$ with $\mu={\rm constant}$, then $\mu=0$ and $M^n$ is
of parallel cubic form.
\end{proposition}

\begin{proof}
We first fix a point $p\in M^n$, and then we choose an orthonormal
basis $\{e_i\}_{i=1}^n$ as in Lemma \ref{lm:3.4} such that
\begin{equation}\label{eqn:5.1}
K_{e_1}e_i=\lambda_i e_i,\ \ i=1,\ldots,n.
\end{equation}
We take a geodesic $\gamma(s)$ passing through $p$ in the direction
of $e_1$. Let $\{E_1,\ldots, E_n\}$ be parallel vector fields along
$\gamma$, such that $E_i(p)=e_i,\ 1\le i\le n$, and
$E_1=\gamma'(s)$. Then we have $h(E_i,E_j)=h(e_i,e_j)=\delta_{ij}$
for $1\le i,j\le n$.

Applying \eqref{eqn:3.9}, we get that
\begin{equation}\label{eqn:5.3}
\tfrac{\partial}{\partial s}h(K(E_1,E_1),E_i,)
=h((\hat{\nabla}_{E_1} K)(E_1,E_1),E_i,)=0,\ \ 2\le i\le n,
\end{equation}
\begin{equation}\label{eqn:5.4}
\tfrac{\partial}{\partial s}h(K(E_1,E_i),E_j) =h((\hat{\nabla}_{E_1}
K)(E_1,E_i),E_j)=0,\ \ 2\le i\neq j\le n.
\end{equation}
Then we have
\begin{equation}\label{eqn:5.5}
\left\{
\begin{aligned}
& h(K(E_1,E_1),E_i)=h(K(e_1,e_1),e_i)=0,\\[1mm]
&h(K(E_1,E_i),E_j)=h(K(e_1,e_i),e_j)=0,
\end{aligned}
\right.\ \ 2\le i\neq j\le n.
\end{equation}
It follows that there exist functions $\tilde{\lambda}_i\ (1\le i\le
n)$ defined along $\gamma$, such that
\begin{equation}\label{eqn:5.6}
K_{E_1}E_i=\tilde{\lambda}_i E_i, \ \tilde{\lambda}_i(p)=\lambda_i,\
\ i=1,\ldots,n.
\end{equation}

Now, due to \eqref{eqn:5.6} and that $\mu={\rm constant}$, we can
follow the proof of \eqref{eqn:3.13} to show that, along $\gamma$,
\begin{equation}\label{eqn:5.7}
(2\tilde{\lambda}_i-\tilde{\lambda}_1)(\tilde{\lambda}_i^2
-\tilde{\lambda}_1\tilde{\lambda}_i +\varepsilon)=0,\ \ 2\le i\le n.
\end{equation}

Applying \eqref{eqn:3.9} again, we get
\begin{equation}\label{eqn:5.8}
\left\{
\begin{aligned}
& \tfrac{\partial}{\partial s}\tilde{\lambda}_1=\tfrac{\partial}{\partial s}h(K(E_1,E_1),E_1)
=h((\hat{\nabla}_{E_1} K)(E_1,E_1),E_1)=3\mu,\\[1mm]
&\tfrac{\partial}{\partial
s}\tilde{\lambda}_i=\tfrac{\partial}{\partial s}h(K(E_1,E_i),E_i)
=h((\hat{\nabla}_{E_1} K)(E_1,E_i),E_i)=\mu,\ \ 2\le i\le n.
\end{aligned}
\right.
\end{equation}

By \eqref{eqn:5.8}, taking the derivative of \eqref{eqn:5.7} three
times along $\gamma(s)$ implies that
\begin{equation}\label{eqn:5.9}
12\mu^3=0.
\end{equation}
This combining with \eqref{eqn:3.9} clearly implies that $M^n$ has
parallel cubic form.
\end{proof}

\vskip 2mm

\numberwithin{equation}{section}

\subsection{Completion of the proof of Theorem
\ref{thm:1.1}}\label{sect:4.3}

As we have already stated in the beginning of Section \ref{sect:4},
to prove Theorem \ref{thm:1.1}, we are left to consider the case
that \eqref{eqn:3.9} holds identically for some function $\mu$ on
$M^n$. Now, we should consider two cases: $\mu\neq {\rm constant}$,
or $\mu= {\rm constant}$.

(1) If $\mu\neq {\rm constant}$, then we can apply Lemma
\ref{lm:4.4} to obtain that
$$
T^1=\tfrac{n+2}n\lambda\not=0,\ \ T^2=\cdots=T^n=0,
$$
which, by \eqref{eqn:4.17}, further implies that
$K_{ij}^k=\tfrac{n}{n+2}(T^k \delta_{ij} +T^i \delta_{jk}+T^j
\delta_{ik})$. This implies that $\tilde{K}=0$. According to
subsection 7.1.1 of \cite{SSV}, and also Lemma 2.1 of \cite{LLSSW}
and noting that $K\neq0$, we easily see that locally $M^n$ is a
hyperquadric, which either has no center, or is not centered at the
origin.

(2) If $\mu= {\rm constant}$, then by Proposition \ref{pr:5.1},
$M^n$ is of parallel cubic form. It follows that we can apply the
(classification) Theorem 1.1 of \cite{CHM} to see that locally $M^n$
is either a hyperquadric with the origin as its center (i.e. $K=0$),
or one of the hypersurfaces as stated from (ii) up to (viii) of
Theorem \ref{thm:1.1}.

We have completed the proof of Theorem \ref{thm:1.1}.\qed

\vskip 1cm


\vskip 3mm

\end{document}